\documentclass[noinfoline]{article}

\RequirePackage[OT1]{fontenc}
\RequirePackage[
amsthm,amsmath,natbib
]{imsart}

\usepackage{graphicx}
\usepackage{latexsym,amsmath}
\usepackage{amsmath,amsthm,amscd}
\usepackage{amsfonts}
\usepackage[psamsfonts]{amssymb}
\usepackage{enumerate}
\usepackage{mathscinet}

\usepackage{url}

\usepackage{tocvsec2}

\startlocaldefs

\theoremstyle{plain} 
\newtheorem{theorem}{Theorem}

\newtheorem{proposition}[theorem]{Proposition}

\theoremstyle{definition} 

\theoremstyle{definition} 

\newtheorem*{ex*}{Example}
\theoremstyle{remark} 

\theoremstyle{remark} 

\newtheorem*{remark*}{Remark}
 
\newcommand{\beqa}{\begin{eqnarray}}
\newcommand{\eeqa}{\end{eqnarray}}
 
\newcommand{\bseq}{\begin{subequations}}
 
\newcommand{\eseq}{\end{subequations}}

\newcommand{\N}{\{1,2,\dots\}}
\renewcommand{\N}{{\mathbb{N}}}

\newcommand{\e}[1]{e^{(#1)}}
\newcommand{\f}[2]{f^{(#1,#2)}}
\newcommand{\aaa}[1]{a^{(#1)}}

\newcommand{\aaaa}[1]{a^{|#1}}

\newcommand{\Si}{\Sigma}

\renewcommand{\P}{\operatorname{\mathsf{P}}} 
\newcommand{\E}{\operatorname{\mathsf{E}}}

\newcommand{\R}{\mathbb{R}}

\newcommand{\G}{\overline{\Phi}}

\newcommand{\vp}{\varepsilon}

\renewcommand{\le}{\leqslant}
\renewcommand{\ge}{\geqslant}

\endlocaldefs

\begin{document}

\begin{frontmatter}

\title{On the supremum of the tails of normalized sums of independent Rademacher random variables}
\runtitle{Supremum of Rademacher tails}

%

\begin{aug}
\author{\fnms{Iosif} \snm{Pinelis}\thanksref{t2}\ead[label=e1]{ipinelis@mtu.edu}
}
  \thankstext{t2}{Supported by NSA grant H98230-12-1-0237}
\runauthor{Iosif Pinelis}


\address{Department of Mathematical Sciences\\
Michigan Technological University\\
Houghton, Michigan 49931, USA\\
E-mail: \printead[ipinelis@mtu.edu]{e1}}
\end{aug}

\begin{abstract}
A well-known longstanding conjecture on the supremum of the tails of normalized sums of independent Rademacher random variables is disproved. A related conjecture, also recently disproved, is discussed.  
\end{abstract}

  
%

\begin{keyword}[class=AMS]
\kwd[Primary ]{60E15}
\kwd[; secondary ]
{60G50}
\end{keyword}


\begin{keyword}
\kwd{probability inequalities}
\kwd{tail probabilities}
\kwd{Rade\-macher random variables}
\kwd{sums of independent random variables}
\end{keyword}

\end{frontmatter}

\settocdepth{chapter}


\settocdepth{subsubsection}




Let $\vp:=(\vp_1,\vp_2,\dots)$ be the sequence of independent Rademacher random variables (r.v.'s), so that $\P(\vp_i=1)=\P(\vp_i=-1)=\frac12$ for all $i$.  
Consider the ``maximal'' tail function 
\begin{equation*}
	M(x):=\sup\big\{\P(a\cdot\vp\ge x)\colon a\in\Si\big\} 
\end{equation*}
of the normalized Rademacher sums $a\cdot\vp:=a_1\vp_1+a_n\vp_2+\cdots$, 
where $x\in\R$ and 
\begin{equation*}
	\Si:=\big\{a=(a_1,a_2,\dots)\in\R^\N\colon a_1^2+a_2^2+\dots=1\big\} 
\end{equation*} 
is the unit sphere in $\ell^2$. 
The behavior of $\P\big(a\cdot\vp\ge x\big)$ as a function of $a$ and $x$ is very complicated; e.g., Fig.~1 in \cite{pin-towards} suggests that even for the fixed $a=\big(\underbrace{\tfrac1{10},\dots,\tfrac1{10}}_{100},0,0,\dots\big)$, the tail function $\P\big(a\cdot\vp\ge \cdot\big)$ behaves rather erratically. 

Yet, a number of features of the maximal tail function $M$ are known. 
An easy consequence of the central limit theorem is that 
\begin{equation}\label{eq:lower}
	M(x)\ge\G(x)
\end{equation}
for all $x\in\R$, where $\G$ is the standard normal tail function. 
Recently an upper bound on $M(x)$ was obtained, which is asymptotically equivalent for $x\to\infty$ to the lower bound $\G(x)$ in \eqref{eq:lower} --- see \cite{G-R_arxiv}, as well as references therein, including ones to applications 
in probability/statistics and combinatorics/optimization/ operations research. 

A well-known conjecture, apparently due to Edelman \cite{portnoy,edelman-comm}, has been that for all $x\in\R$ 
\begin{equation}\label{eq:conjec}
	M(x)\overset{(?)}=M^=(x), 
\end{equation}
where 
\begin{equation*}
M^=(x):=\sup_{n\in\N}\P\big(\e n\cdot\vp\ge x\big)\quad\text{and}\quad 
\e n:=\big(\underbrace{\sqrt{\tfrac1n},\dots,\sqrt{\tfrac1n}}_n,0,0,\dots\big),  	
\end{equation*}
so that $\e n\in\Si$ for all $n$ and hence $M(x)\ge M^=(x)$ for all $x\in\R$. 
Thus, the conjecture could be restated as $M(x)\overset{(?)}\le M^=(x)$ for all $x\in\R$. 

As apparently most other people working in the area, I had believed this conjecture to be true ---  
until hearing 
recently that it had been disproved by A.~V.~Zhubr. 
Then I wrote to him to request further information. 
While waiting for a response 
and being aided by the newly acquired belief that 
conjecture \eqref{eq:conjec} is false, I quickly happened to find a counterexample to the conjecture, using the following heuristics, which may be of interest to readers.  
Natural and closest competitors of the ``equal-weight'' sequences $a=\e n$ in terms of having a greater value of $\P(a\cdot\vp\ge x)$ are sequences $a$ of the form 
\begin{equation}\label{eq: f nt}
	\f nt:=\big(\underbrace{\sqrt{\tfrac{1-t^2}n},\dots,\sqrt{\tfrac{1-t^2}n}}_n,t,0,0,\dots\big)\in\Si \quad\text{for}\quad t\in(0,1)   
\end{equation}
and $n\in\N$ --- 
cf.\ e.g.\ \cite[Proof of Proposition~2]{maximal}, where a conjecture presented in \cite{g-peskir} was somewhat similarly disproved; \cite[Proof of Theorem~4.2]{BGH}; and \cite[Proof of Theorem~2]{pin-towards}. 
Clearly, for all $t\in(0,1)$, 
\begin{equation}\label{eq:kolm}
	\P\big(\f nt\cdot\vp\ge x\big)=\tfrac12\P\big(\e n\cdot\vp\ge u\big)+\tfrac12\P\big(\e n\cdot\vp\ge v\big), 
\end{equation}
where 
\begin{equation}\label{eq:u,v}
	u:=\frac{x-t}{\sqrt{1-t^2}}\quad\text{and}\quad v:=\frac{x+t}{\sqrt{1-t^2}}. 
\end{equation} 
Note that for any $x\in(0,1]$ one has $M(x)=M^=(x)=\P(\vp_1\ge x)=\frac12$, so that \eqref{eq:conjec} holds. 
So, as least as far as positive values of $x$ are concerned, one may assume $x>1$.  
Note next that the tail of the r.v. $\e n\cdot\vp$ decreases fast, just as the corresponding normal tail does (and even faster, in a sense, since it simply vanishes in a neighborhood of $\infty$). Therefore, of the two tail values in \eqref{eq:kolm}, $\P\big(\e n\cdot\vp\ge u\big)$ and $\P\big(\e n\cdot\vp\ge v\big)$, the first one is greater, and it may be much greater if $t\in(0,1)$ is not too close to $0$. 
Thus, for a given value of $x$, to make the tail value $\P\big(\f nt\cdot\vp\ge x\big)$ compete with $M^=(x)$, 
one may try to focus on making $\P\big(\e n\cdot\vp\ge u\big)$ large. 
Since the inequality $\e n\cdot\vp\ge u$ is non-strict, it appears to make sense to choose $x$ and $t$ in \eqref{eq:u,v} so that $u$ be one of the atoms of the distribution of the r.v.\ $\e n\cdot\vp$. Thus, assume that $u=k/\sqrt n$ for some $k\in\{0,\dots,n\}$, whence  
\begin{equation}\label{eq:eq}
	x=t+\tfrac k{\sqrt n}\,\sqrt{1-t^2}. 
\end{equation} 
Similarly, to make $\sup_{m\in\N}\P\big(\e m\cdot\vp\ge x\big)=M^=(x)$ comparatively small, one may want to choose $x$ other than the rightmost atom, $\sqrt m$, of the distribution of the r.v.\ $\e m\cdot\vp$, for any $m\in\N$; cf.\ e.g.\ Lemma~1 in \cite{pin-towards} and its (very short) proof therein; 
moreover then, one 
may try to choose $x$ to be about as far away as possible from any of these rightmost atoms, $\sqrt m$. 
Thus, it appears to make sense to try the values $x=\sqrt{\frac{14}{10}},\sqrt{\frac{24}{10}},\sqrt{\frac{34}{10}},\dots$ --- each of these values taken together with all small enough natural $n$; $k\in\{0,\dots,n\}$; and the corresponding values of $t$ found, for each such triple $(x,n,k)$, as a root of equation \eqref{eq:eq}. 

This approach indeed results in a quadruple $(x,n,k,t)$ with 
$x=\sqrt{\frac{74}{10}}=\sqrt{\frac{37}5}$, $n=10$, $k=8$, and $t=\sqrt{\frac5{37}}$, 
which disproves conjecture \eqref{eq:conjec}: 

\begin{proposition}\label{prop:}
For $n=10$, $x=\sqrt{\tfrac{37}5}$, and $t=\sqrt{\tfrac5{37}}$, 
\begin{equation}\label{eq:main}
		M(x)\ge
		\P\big(\f nt\cdot\vp\ge x\big)
		>M^=(x). 
\end{equation}
\end{proposition}

In the hindsight, it is hardly a coincidence that for the disproving triple $(n,x,t)=\big(10,\sqrt{\tfrac{37}5},\sqrt{\tfrac5{37}}\big)$ one has $t=1/x$. Indeed, for a given $u=k/\sqrt n$, it makes sense to choose $x$ to be as large as possible, in order to make $M^=(x)$ as small as possible. 
That is, one can try to take the largest value of $x$ for which equation \eqref{eq:eq} has a solution $t\in(0,1)$. Thus, one obtains 
\begin{equation*}
	x=\sqrt{1+k^2/n}\quad\text{and}\quad t=1/x.
\end{equation*}

To prove Proposition~\ref{prop:}, we shall need the Berry--Esseen inequality 
\begin{equation}\label{eq:BE}
	\Big|\P\Big(\frac1{\sqrt n}\sum_1^n X_i\ge x\Big)-\G(x)\Big|\le\frac{C}{\sqrt n}\,\E|X_1|^3  
\end{equation}
for all $n\in\N$ and $x\in\R$, where $X_1,X_2,\dots$ are independent identically distributed zero-mean unit-variance r.v.'s and $C$ is an absolute constant. 
The apparently smallest currently known value of $C$, equal to $0.4748$, is due to Shevtsova~\cite{shev11}; a slightly larger value, $0.4785$, was established earlier by Tyurin~\cite{tyurin}. 

\begin{proof}[Proof of Proposition~\ref{prop:}]
The first inequality in \eqref{eq:main} follows immediately from the definitions of $M(x)$ and $\f nt$, since $\f nt\in\Si$. 
Next, take indeed $n=10$, $x=\sqrt{\tfrac{37}5}$, and $t=\sqrt{\tfrac5{37}}$. 
Then, by \eqref{eq:BE} 
with the mentiond constant factor $C=0.4748$, for all natural $j>50640$ 
\begin{equation*}
	\P\big(\e j\cdot\vp\ge x\big)\le\G(x)+\tfrac{0.4748}{\sqrt{50641}}<\P\big(\f nt\cdot\vp\ge x\big). 
\end{equation*}
On the other hand, 
it is straightforward to check that 
\begin{equation*}
	\P\big(\f nt\cdot\vp\ge x\big)>\max_{1\le j\le50640}\P\big(\e j\cdot\vp\ge x\big)  
\end{equation*}
(which takes about 20 minutes to do using Mathematica).   
Recalling now the definition of $M^=(x)$, one obtains the second inequality in \eqref{eq:main} as well. 
\end{proof} 

Soon after finding the counterexample described above and informing \break 
A.~V.~Zhubr about it, I received a response from him with a copy of his paper \cite{zhubr}. 
It turns out that in it, not conjecture \eqref{eq:conjec}, but the following related ``finite-dimensional'' conjecture was disproved: 
\begin{equation}\label{eq:conj n}
	M_n(x)\overset{(?)}=M^=_n(x) 
\end{equation}
for all $n\in\N$ and $x\in\R$, 
where 
\begin{align*}
	M_n(x)&:=\max\big\{\P(a\cdot\vp\ge x)\colon a=(a_1,\dots,a_n,0,0,\dots)\in\Si\big\} \quad\text{and}\\ 
	M^=_n(x)&:=\max_{j\in\N,\;j\le n}\P\big(\e j\cdot\vp\ge x\big). 
\end{align*}

How do conjectures \eqref{eq:conjec} and \eqref{eq:conj n} relate to each other? 

To approach an answer to this question, let us start by taking any $x\in\R$. Then take any sequence $a=(a_1,a_2,\dots)\in\Si$. 
For each large enough $n\in\N$
, consider the ``truncated'' version $\aaaa n:=(a_1,\dots,a_n,0,0,\dots)/\sqrt{a_1^2+\dots+a_n^2}$ of the sequence $a$, so that 
$\aaaa n\in\Si$ and 
$\aaaa n\cdot\vp\to a\cdot\vp$ almost surely as $n\to\infty$. Hence, by the Fatou lemma, 
\begin{equation}\label{eq:fatou}
	\liminf_{n\to\infty}\P\big(\aaaa n\cdot\vp\ge x\big)\ge\P\big(a\cdot\vp\ge x\big). 
\end{equation}
In particular, it follows that 
\begin{equation}\label{eq:approx}
	M_n(x)\uparrow M(x) 
\end{equation}
for each $x\in\R$ 
as $n\to\infty$. 

More specifically, suppose now that the real number $x$ and the sequence $a=(a_1,a_2,\dots)\in\Si$ disprove conjecture \eqref{eq:conjec} in the sense that $\P\big(a\cdot\vp\ge x\big)>M^=(x)$ -- cf.\ \eqref{eq:main}. 
Then, by \eqref{eq:fatou}, $\P\big(\aaaa n\cdot\vp\ge x\big)>M^=(x)$ for all large enough $n\in\N$. 
Therefore and because $M^=(x)\ge M^=_n(x)$, it follows that, if conjecture \eqref{eq:conjec} is disproved, then conjecture \eqref{eq:conj n} is disproved as well.  

However, we shall now see that the vice versa implication is not true, in the sense that one can find some $n\in\N$, $a=(a_1,\dots,a_n,0,0,\dots)\in\Si$, and $x\in(0,\infty)$ such that $\P(a\cdot\vp\ge x)>M^=_n(x)
$ and yet $\P(a\cdot\vp\ge x)<M^=(x)$; in fact, one can find such $a$ and $x$ for every large enough $n\in\N$.  

Indeed, 
for each natural $n\ge8$ let 
\begin{equation}\label{eq:series}
	\aaa n:=\f n{t_n}\quad\text{and}\quad t_n:=1/x_n,\quad\text{where}\quad x_n:=\sqrt{n-3+4/n}    
\end{equation}
and $\f nt$ is as in \eqref{eq: f nt}; 
then, as shown in \cite[Theorem~3 and Corollary~4]{zhubr}, 
\begin{equation*}
	M_n(x_n)\ge\P(\aaa n\cdot\vp\ge x_n)=(n+1)/2^{n+1}>
	M^=_n(x_n),  
\end{equation*}
which disproves conjecture \eqref{eq:conj n}.  
This series of counterexamples 
was obtained in \cite{zhubr} based on certain geometric, rather than probabilistic, considerations (and using rather different notations). 

However, for all large enough $n$ and $\aaa n$ as in \eqref{eq:series}, 
\begin{equation}\label{eq:less} 
\P(\aaa n\cdot\vp\ge x_n)<M^=(x_n). 	
\end{equation}
Indeed, 
for $n\to\infty$ one has $\P(\aaa n\cdot\vp\ge x_n)=(n+1)/2^{n+1}=2^{-n(1+o(1))}$, which is asymptotically much less than $\G(x_n)=\exp\{-x_n^2/(2+o(1))\}=(\sqrt e\,)^{-n(1+o(1))}$; 
in fact, it is rather easy (even if a bit tedious) to see that $\P(\aaa n\cdot\vp\ge x_n)<\G(x_n)$ for all natural $n\ge16$ 
and hence 
\eqref{eq:less} holds 
for such $n$ --- because, again by the central limit theorem (cf.\ \eqref{eq:lower}), $M^=(x)\ge\G(x)$ for all $x\in\R$. 

On the other hand --- for $n=10$, and $x_n$ and $t_n$ as in \eqref{eq:series} --- one has 
$x_n=\sqrt{\tfrac{37}5}$ and hence $t_n=\sqrt{\tfrac5{37}}$, so that (as was pointed out to me by A.\ V.\ Zhubr) the example given in Proposition~\ref{prop:} turns out to belong to the series of examples defined in \eqref{eq:series}. 
Moreover, one can check, just as easily as in the proof of Proposition~\ref{prop:}, that \eqref{eq:main} will hold with 
the triple $(n,x_n,t_n)=\big(8,\sqrt{\frac{11}2},\sqrt{\frac2{11}}\big)$ or $(n,x_n,t_n)=\big(9,\sqrt{\frac{58}9},\sqrt{\frac9{58}}\big)$ in place of the triple $(n,x,t)=\big(10,\sqrt{\tfrac{37}5},\sqrt{\tfrac5{37}}\big)$ used in Proposition~\ref{prop:}. 
In fact, each of these two checks, for $n=8$ and $n=9$, requires much less computer time, respectively about 4 seconds and one minute  instead of 20 minutes, since one can then use much smaller threshold values: $3461$ for $n=8$ and $12775$ for $n=9$, in place of the threshold value $50640$ for $n=10$ used in the proof of Proposition~\ref{prop:}. 
Quite possibly, \eqref{eq:main} may similarly hold for 
some other triples of the form $(n,x_n,t_n)$ with $x_n$ and $t_n$ as in \eqref{eq:series} and  $n\in\{11,\dots,15\}$; however,  checking that for any one of these 5 values of $n$ seems to require too much computer time.  
On the other hand, as was noted, \eqref{eq:main} will not hold for 
any such triples $(n,x,t)=(n,x_n,t_n)$ with $n\in\{16,17,\dots\}$.  

One may also note that, in view of \eqref{eq:approx}, conjecture \eqref{eq:conjec} can be restated in a quasi-finite-dimensional form, as 
$\sup_{n\in\N}M_n(x)\overset{(?)}=M^=(x)$ for all $x\in\R$. 

Observe that $M^=(x)$ is the supremum of $\P(a\cdot\vp\ge x)$ over all sequences $a\in\Si$ taking only one nonzero value. 
A question that remains open here is whether, for all $x\in\R$,  
the supremum \big(equal by definition to $M(x)$\big) of $\P(a\cdot\vp\ge x)$ over all sequences $a\in\Si$ is the same as that over all sequences $a\in\Si$ taking 
exactly (or, equivalently, at most) two nonzero values (cf.\ \cite[Theorem~6]{zhubr}). 
More generally, one may ask whether, for some $q\in\{2,3,\dots\}$ and all $x\in\R$, the value of $M(x)$ is equal to the supremum of $\P(a\cdot\vp\ge x)$ over all sequences $a\in\Si$ 
taking 
exactly (or, equivalently, at most) $q$  
nonzero values. 


\bibliographystyle{abbrv}


{\footnotesize
\bibliography{C:/Users/Iosif/Dropbox/mtu/bib_files/citations}

\def\cprime{$'$} \def\polhk#1{\setbox0=\hbox{#1}{\ooalign{\hidewidth
  \lower1.5ex\hbox{`}\hidewidth\crcr\unhbox0}}}
  \def\polhk#1{\setbox0=\hbox{#1}{\ooalign{\hidewidth
  \lower1.5ex\hbox{`}\hidewidth\crcr\unhbox0}}}
  \def\polhk#1{\setbox0=\hbox{#1}{\ooalign{\hidewidth
  \lower1.5ex\hbox{`}\hidewidth\crcr\unhbox0}}} \def\cprime{$'$}
  \def\polhk#1{\setbox0=\hbox{#1}{\ooalign{\hidewidth
  \lower1.5ex\hbox{`}\hidewidth\crcr\unhbox0}}}
  \def\polhk#1{\setbox0=\hbox{#1}{\ooalign{\hidewidth
  \lower1.5ex\hbox{`}\hidewidth\crcr\unhbox0}}}
\begin{thebibliography}{10}

\bibitem{BGH}
S.~G. Bobkov, F.~G{\"o}tze, and C.~Houdr{\'e}.
\newblock On {G}aussian and {B}ernoulli covariance representations.
\newblock {\em Bernoulli}, 7(3):439--451, 2001.

\bibitem{edelman-comm}
D.~Edelman.
\newblock Private communication, 1994.

\bibitem{g-peskir}
S.~E. Graversen and G.~Pe{\v{s}}kir.
\newblock Extremal problems in the maximal inequalities of {K}hintchine.
\newblock {\em Math. Proc. Cambridge Philos. Soc.}, 123(1):169--177, 1998.

\bibitem{G-R_arxiv}
I.~Pinelis.
\newblock An asymptotically {G}aussian bound on the {R}ademacher tails,
  preprint, \url{http://arxiv.org/find/all/1/au:+pinelis/0/1/0/all/0/1}.

\bibitem{maximal}
I.~Pinelis.
\newblock On exact maximal {K}hinchine inequalities.
\newblock In {\em High dimensional probability, {II} ({S}eattle, {WA}, 1999)},
  volume~47 of {\em Progr. Probab.}, pages 49--63. Birkh\"auser Boston, Boston,
  MA, 2000.

\bibitem{pin-towards}
I.~Pinelis.
\newblock Toward the best constant factor for the {R}ademacher-{G}aussian tail
  comparison.
\newblock {\em ESAIM Probab. Stat.}, 11:412--426, 2007.

\bibitem{portnoy}
S.~Portnoy.
\newblock Private communication, 1991.

\bibitem{shev11}
I.~Shevtsova.
\newblock On the absolute constants in the {B}erry-{E}sseen type inequalities
  for identically distributed summands.
\newblock \url{http://arxiv.org/abs/1111.6554}, 2011.

\bibitem{tyurin}
I.~Tyurin.
\newblock New estimates of the convergence rate in the {L}yapunov theorem
  (preprint, ar{X}iv:0912.0726v1 [math.{PR}]).

\bibitem{zhubr}
A.~V. Zhubr.
\newblock On one extremal problem for ${N}$-cube.
\newblock To appear, 2012.

\end{thebibliography}
}

\end{document}